\documentclass[10pt]{amsart}
\usepackage{graphicx}
\newtheorem{theorem}{Theorem}[section]

\theoremstyle{definition}

\theoremstyle{remark}

\newcommand{\Z}{\mathbb{Z}}

\newcommand{\R}{\mathbb{R}}

\newcommand{\st}{\mid}

\numberwithin{equation}{section}

\begin{document}

\title{Riemann-Roch theory on finite sets}

\author{Rodney James}
\address{Deptartment of Mathematical and Statistical Sciences, University of Colorado, Denver, CO, USA}

\author{Rick Miranda}
\address{Department of Mathematics, Colorado State University, Fort Collins, CO, USA}


\keywords{}

\date{}

\dedicatory{}

\begin{abstract}
In \cite{BN} M. Baker and S. Norine developed a theory of divisors and linear systems on graphs, and proved a Riemann-Roch Theorem for these objects (conceived as integer-valued functions on the vertices).  In \cite{JM09} and \cite{JM11} the authors generalized these concepts to real-valued functions, and proved a corresponding Riemann-Roch Theorem in that setting, showing that it implied the Baker-Norine result.  In this article we prove a Riemann-Roch Theorem in a more general combinatorial setting that is not necessarily driven by the existence of a graph.
\end{abstract}

\maketitle


\section{Introduction}

Baker and Norine showed in~\cite{BN}
that a Riemann-Roch formula holds for an analogue of linear systems
defined on the vertices of finite connected graphs.
There, the image of the graph Laplacian
induces a equivalence relation on the group of \emph{divisors} of the graph,
which are integer-valued functions defined on the set of vertices.
This equivalence relation is the analogue of linear equivalence in the classical algebro-geometric setting.

We showed in~\cite{JM09} that the Baker-Norine result
implies a generalization of the Riemann-Roch formula to edge-weighted graphs,
where the edge weights can be $R$-valued, where $R$ is an arbitrary subring of the reals;
the equivalence relation induced by image of the edge-weighted graph Laplacian
applies equally well to divisors which are $R$-valued functions defined on the set of vertices.
In~\cite{JM11}, we proved our version of the $R$-valued Riemann-Roch theorem from first principles;
this gave an independent proof of the Bake-Norine result as well.

The notion of linear equivalence above
is induced by the appropriate graph Laplacian acting on the group of divisors,
which may be viewed as points in $\Z^{n}$ (the Baker-Norine case) or more generally $R^{n}$,
where $n$ is the number of vertices of the graph.
In this paper, we propose a generalization of this Riemann-Roch formula for graphs
where linear equivalence is induced by a group action on the points of $R^{n}$.  
The setup we will use is as follows.

Choose a subring $R$ of the reals, and fix a positive integer $n$.
Let $V$ be the group of points in $R^{n}$ under component-wise addition.  
If $x \in V$, we we will use the functional notation $x(i)$ to denote the the $i$-th component of $x$. 

For any  $x \in V$, define the \emph{degree} of $x$ as
\[
\deg(x) = \sum_{i=1}^{n} x(i).
\]
For any $d \in R$, define the subset $V_{d} \subset V$ to be
\[ 
V_{d} = \{ x \in V \st \deg(x)=d \}.
\]
Note that the subset $V_0$ is a subgroup;
for any $d$, $V_d$ is a coset of $V_0$ in $V$.

Let $H \subset V_0$ be a subgroup of $V_{0}$
and consider the action on $V$ by $H$ by translation: 
if $h \in H$ and $x \in V$, then $(h,x) \mapsto h+x$.
This action of $H$ on $V$ induces the equivalence relation $x \sim y$ if and only if $x-y \in H$;
or equivalently, $x \sim y$ if and only if there is a $h \in H$ such that $x = (h,y)$.

Fix the parameter $g\in R$, which we call the \emph{genus},
and choose a set $\mathcal{N} \subset V_{g-1}$.
For $x \in V$, define 
\begin{eqnarray*}
x^{+} &=& \max(x,0) \\
x^{-} &=& \min(x,0)
\end{eqnarray*}
where $\max$ and $\min$ are evaluated at each coordinate.
It follows that $x = x^{+} + x^{-}$ and $x^{+} =  -(-x)^{-}$.
We then define the \emph{dimension} of $x \in V$ to be
\[
\ell(x) = \min_{\nu \in \mathcal{N}} \{ \deg((x-\nu)^{+}) \}.
\]
We will discuss the motivation for this definition of dimension in the next section.

We can now state our main result.
\begin{theorem} \label{RR}
Let $V$ be the additive group of points in $R^{n}$ for a subring $R \subset \R$,
and let $H$ be a subgroup of $V_{0}$ that acts by translation on $V$.
Fix $g \in R$.
Suppose $\kappa \in V_{2g-2}$, and $\mathcal{N} \subset V_{g-1}$,
satisfying the symmetry condition
\[
\nu \in \mathcal{N} \Longleftrightarrow \kappa - \nu \in \mathcal{N}.
\]
Then for every $x \in V$,
\[
\ell(x) - \ell(\kappa - x) = \deg(x) - g +1
\]
\end{theorem}

We will give a proof of Theorem~\ref{RR} in \S 2.
In \S 3, we will give examples of $H$, $\kappa$, and $\mathcal{N}$
(coming from the graph setting)
which satisfy the conditions of Theorem~\ref{RR},
and show how this Riemann-Roch formulation is equivalent to that given in~\cite{JM11}.
Finally in \S 4, we gives examples that do not arise from graphs.


\section{Proof of Riemann-Roch Formula}
The dimension of $x \in V$ 
\[
\ell(x) = \min_{\nu \in \mathcal{N}} \{ \deg((x-\nu)^{+}) \}
\]
can be written as
\[
\ell(x) =\min_{\nu \in \mathcal{N}} \left\{ \sum_{i=1}^{n} \max\{x(i)-\nu(i),0\} \right\}.
\]
If $x(i) \ge \nu(i)$ for each $i$,
$\sum_{i=1}^{n} \max\{x(i)-\nu(i),0\}$ is the \emph{taxicab} distance from $x$ to $\nu$.
Thus, $\ell(x)$ is the taxicab distance from $x$
to the portion of the set $\mathcal{N}$ such that $x \ge \mathcal{N}$,
where the inequality is evaluated at each component.  

We will now proceed with the proof of the Riemann-Roch formula.
\begin{proof}{(Theorem~\ref{RR})}

Suppose that $\mathcal{N} \subset V_{g-1}$ and $\kappa \in V$
satisfy the symmetry condition.
We can then write
\begin{eqnarray*}
\ell(\kappa - x) &=& \min_{\nu \in N  } \{\deg((\kappa-x - \nu)^{+}) \} \\
                 &=& \min_{\nu \in N }  \{\deg(((\kappa-\nu) -x )^{+}) \} \\
                 &=& \min_{\mu \in N} \{ \deg((\mu - x)^{+}) \}. 
\end{eqnarray*}
Using the identities $x = x^{+} + x^{-}$ and $x^{+} = -(x^{-})$, we have
\begin{eqnarray*}                 
\min_{\mu \in N} \{ \deg((\mu - x)^{+}) \} &=& \min_{\mu \in N} \{ \deg((\mu - x)) - \deg((\mu - x)^{-}) \} \\
                 &=& \min_{\mu \in N} \{ \deg((\mu - x)) + \deg((x-\mu)^{+}) \}.
\end{eqnarray*}
Since $\mu \in \mathcal{N}$ we know that $\deg(\mu)=g-1$, thus $\deg(\mu - x) = g-1 - \deg(x)$ and thus
\begin{eqnarray*}
\ell(\kappa - x) &=&\deg((\mu - x))+\min_{\mu \in N} \{  \deg((x-\mu)^{+}) \} \\
                 &=& g-1 - \deg(x) + \ell(x).
\end{eqnarray*}
\end{proof}
Note that $\kappa \in V$ is the analogue to the canonical divisor in the classical Riemann-Roch formula.


\section{Graph Examples}

Let $\Gamma$ be a finite, edge-weighted connected simple graph with $n$ vertices.
We will assume that $\Gamma$ has no loops.
Let $w_{ij} \in R$ with $w_{ij} \ge 0$ be the weight of the edge connecting vertices $v_{i}$ and $v_{j}$.
The no loops assumption is also applied to the edge weights so that $w_{ii}=0$ for each $i$.
We showed in~\cite{JM11} that such a graph satisfies an equivalent Riemann-Roch formula as in Theorem~\ref{RR}.  

In this setting, $H = < h_{1}, h_{2}, \ldots, h_{n-1} > $ where each $h_{i} \in R^{n}$ is defined as 
\[ 
h_{i}(j) = \left\{ \begin{array}{ll} \deg(v_{i}) & \mbox{if } i = j \\ -w_{ij} & \mbox{if } i \ne j. \end{array} \right.
\]
(Here $\deg(v)$ for a vertex $v$ is the sum of the weights of the edges incident to $v$.)
Note that $H$ is the edge-weighted Laplacian of $\Gamma$. 

As shown in~\cite{JM11}, the set $\mathcal{N} \subset V_{g-1}$
is generated by a set $\{\nu_{1}, \ldots, \nu_{s}\}$ as follows.
Fix a vertex $v_{k}$ and let $(j_{1}, \ldots, j_{n})$ be a permutation of $(1, \ldots, n)$ such that
 $j_{1}=k$.
There are then $(n-1)!$ such permutations;
for each permutation, we compute a $\nu \in V_{g-1}$ defined by
\[
\nu(j_{l}) = \left\{ \begin{array}{ll} -1 & \mbox{if } l=1 \\ 
-1 + \sum_{i=1}^{l-1} w_{j_{i}j_{l}} & \mbox{if } l >1. \end{array} \right. 
\]
Each such $\nu$ may not be unique;
set $s$ to be the number of unique $\nu$'s
and index this set $\{\nu_{1}, \ldots, \nu_{s}\}$.
We then define the set $\mathcal{N}$ as
\[
\mathcal{N} = \{ x \in V \st x \sim \nu_{i} \mbox{ for some } i=1, \ldots s \}.  
\]

The canonical element $\kappa$ is defined by $\kappa(j) = \deg(v_{j})-2$, 
and the genus $g = 1 + \sum_{i<j} w_{ij} - n$.

As an example, consider a two vertex graph $\Gamma$ with edge weight $w_{12} = p >0$.
Then $g=p-1$ and $H = < (p,-p)>= \Z (p,-p)$. 
The set $\mathcal{N} \subset V_{g-1}$ is 
$\mathcal{N}= \{ \nu \st \nu \sim (p-1,-1) \}$ and $\kappa = (p-2,p-2)$. 
Figure~\ref{fig1} shows the divisors $x \in \R^{2}$ for this graph in the plane.
The shaded region, which is bounded
by the corner points in the set $\mathcal{N}$, represent points $x$ with $\ell(x)=0$.  
\begin{figure}
\includegraphics[width=3in]{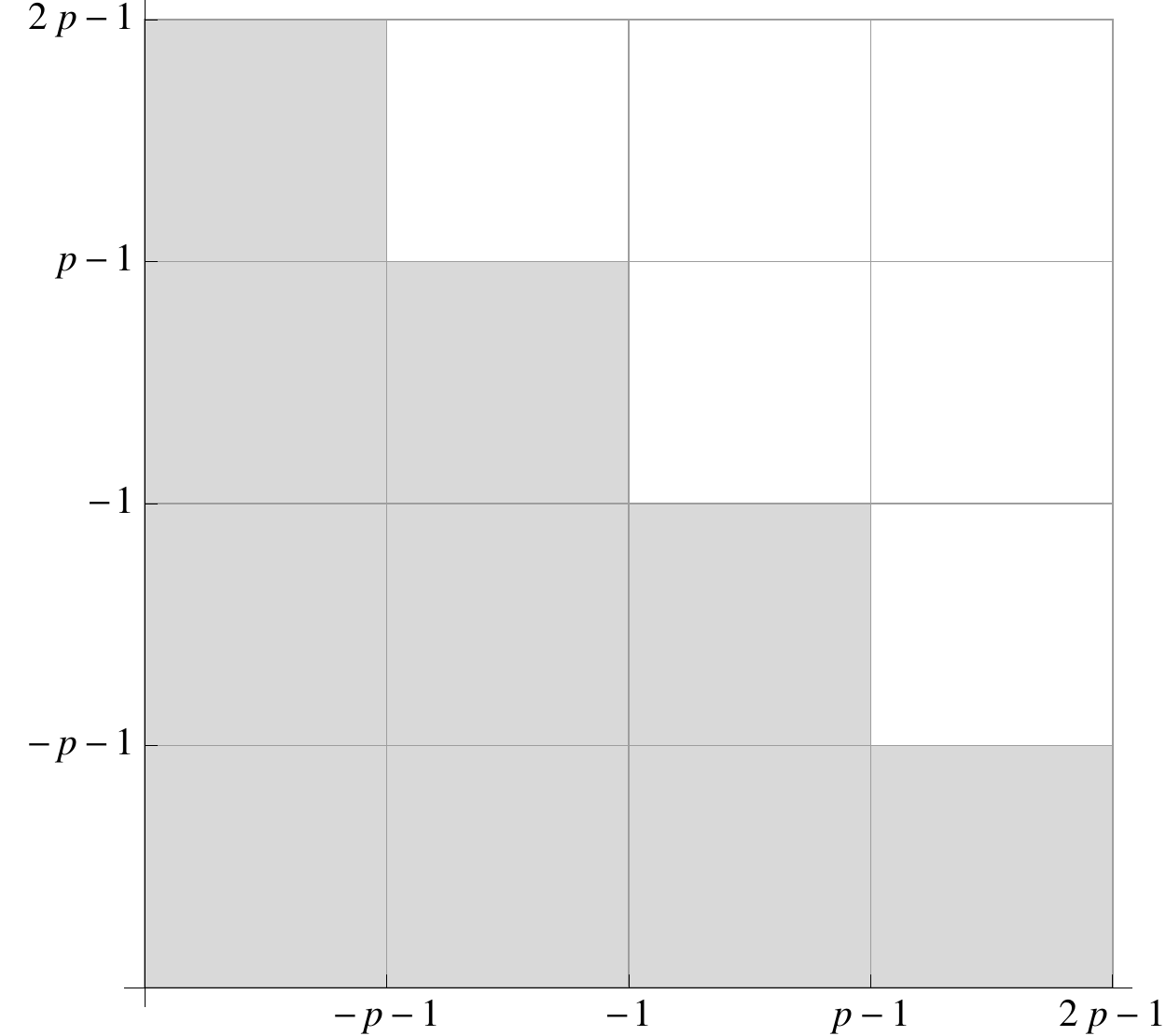}
\caption{Divisors in $\R^{2}$ for a two-vertex graph with $p$ edges.  The shaded region represents points $x \in \R^{2}$ with $\ell(x)=0$; for a general point $x \in \R^{2}$, $\ell(x)$ is the taxicab distance to the shaded region.} \label{fig1}
\end{figure}

We can show directly that $\mathcal{N}$ and $\kappa$ for the two-vertex graph example satisfy the necessary condition for Theorem~\ref{RR} to hold.
If $\kappa -x \in \mathcal{N}$, then $\kappa -x = m(p,-p) + (p-1,-1)$ for some $m \in \Z$.
Solving for $x$, we have 
\begin{eqnarray*}
x &=& (p-2,p-2)-m(p,-p)-(p-1,-1) \\
  &=& (-mp-1,mp+p-1) \\
  &=& (p-1,-1) - (m+1)(p,-p)
\end{eqnarray*}
and thus $x \in \mathcal{N}$.
Similarly, if $\nu \in \mathcal{N}$, it easily follows that $\kappa - \nu \in \mathcal{N}$.

Now consider a three vertex graph with edge weights $w_{12}=p$, $w_{13}=q$ and $w_{23}=r$.  The set $\mathcal{N}$
can be generated by $\nu_{1} = (-1,p-1,q+r-1)$ and $\nu_{2}=(-1,p+r-1,q-1)$;
$H$ can be generated by $h_{1}=(p+q,-p,-q)$ and $h_{2}=(-p,p+r,-r)$.  
In Figure~\ref{fig2}, the region representing points $x \in \R^{3}$ such that $\ell(x)=0$ is
shown for a three vertex graph with edge weights $p=1$, $q=3$ and $r=4$.  
\begin{figure}
\includegraphics[width=4in]{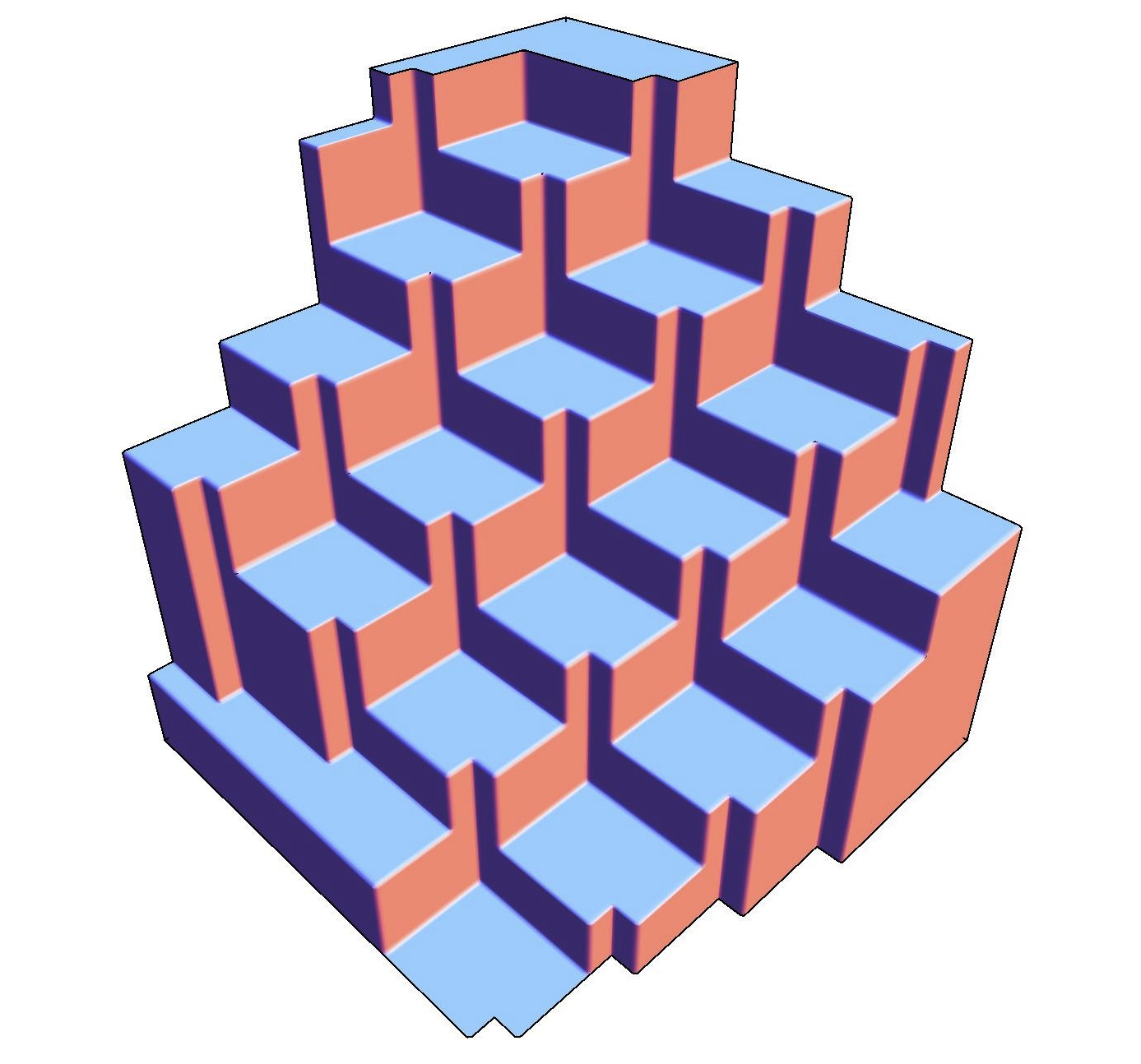}
\caption{Divisors in $\R^{3}$ for a three-vertex graph edge weights $w_{12}=1$,
 $w_{13}=3$ and $w_{23}=4$.  The solid region represents points $x \in \R^{3}$ with $\ell(x)=0$; 
for a general point $x \in \R^{3}$, $\ell(x)$ is the taxicab distance to the surface.
} \label{fig2}
\end{figure}


\section{Non-graph Examples}
The main result of this paper would not be interesting if there were no examples of subgroups $H \subset V_{0}$
with $\mathcal{N}$ and $\kappa$ that were not derived from graphs. 
\begin{theorem} \label{NonGraph}
There are subgroups $H \subset V_{0}$ with $\kappa \in V$ and $\mathcal{N} \subset V_{g-1}$
such that Theorem~\ref{RR}
holds where $H$ is not the Laplacian of a finite connected graph.
\end{theorem}
\begin{proof}

Let $n=2$ and choose $H=<(-4,4)>$, $\mathcal{N}=\{ \nu \in G \st \nu \sim (2,-2) \}$ with $\kappa=(0,0)$ and $g=1$.  
If $H$ were generated from a two vertex graph,
using the notation from the previous section we would have $p=4$.
This would require $\kappa=(2,2)$ with $\mathcal{N}$ generated by $\nu_{1}=(3,-1)$.  

Since there is no integer $m$ such that $\kappa=(0,0) = (2,2) + m(-4,4)$ 
(and likewise there is no $m$ such that $\nu_{1} = (2,-2) = (3,-1) + m(-4,4)$),
$H$ cannot be generated from a two vertex graph.

Now suppose that $\nu \in \mathcal{N}$.
Then $\nu = (2,-2) + m(-4,4)$ for some integer $m$, and 
\begin{eqnarray*}
\kappa - \nu &=& (0,0) - (2,-2) - m(-4,4)\\
             &=& (2,-2) -(m-1)(-4,4)
\end{eqnarray*}
thus $\kappa - \nu \in \mathcal{N}$.  

Similarly, if $\kappa - \nu \in \mathcal{N}$,
$\kappa - \nu = (2,-2) + m(-4,4)$ for some integer $m$, and
\begin{eqnarray*}
\nu &=& \kappa -(2,-2)-m(-4,4) \\
    &=& (4m-2,-4m+2) \\
    &=& (2,-2) - (m-1)(-4,4)
\end{eqnarray*}
thus $\nu \in \mathcal{N}$ and $H$, $\kappa$, $\mathcal{N}$ satisfies Theorem~\ref{RR}.    
\end{proof}
\begin{figure}
\includegraphics[width=3in]{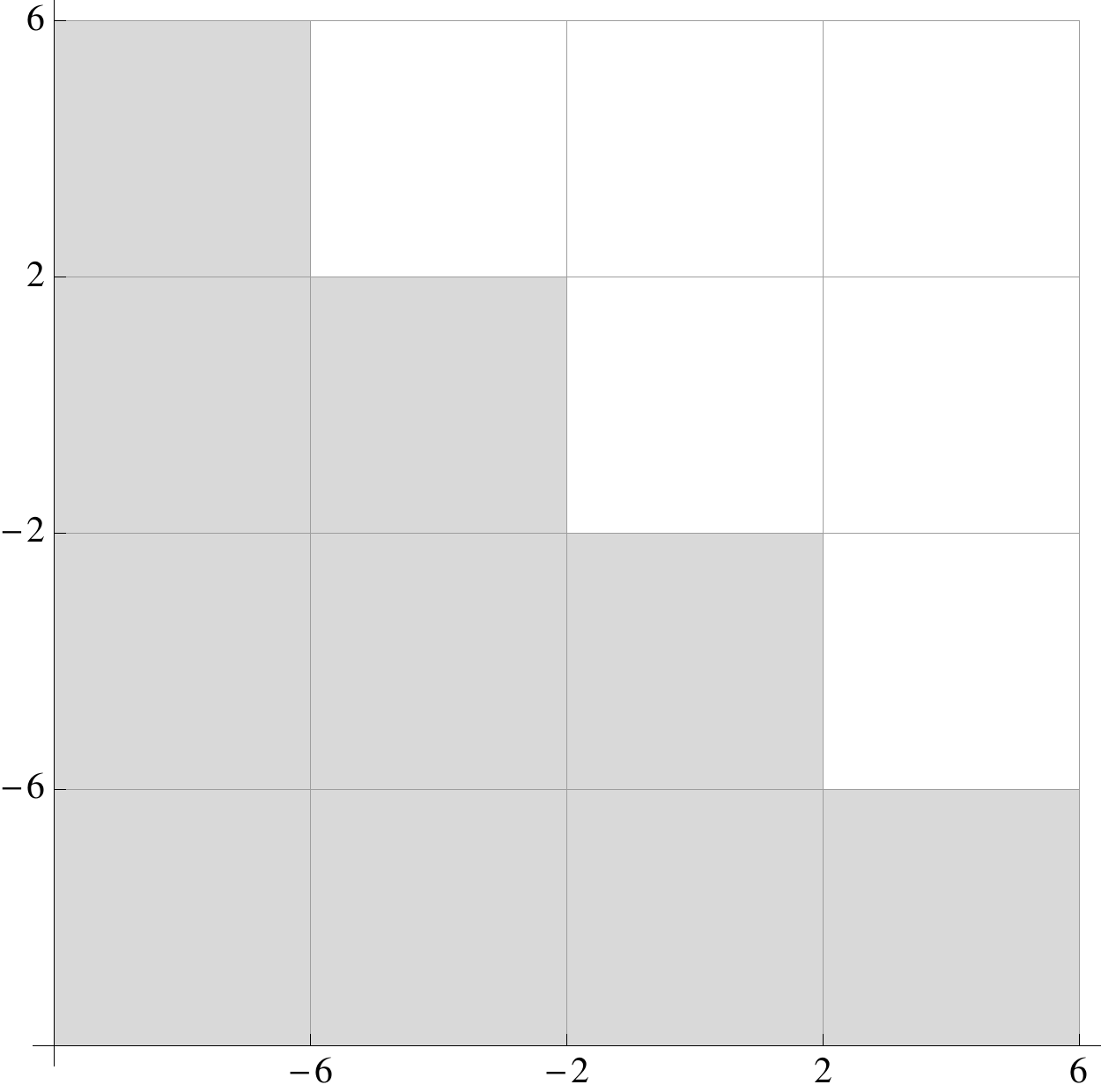}
\caption{Divisors $x \in \R^{2}$ with $\ell(x)=0$ for the non-graph example in the proof of 
Theorem~\ref{NonGraph}.  Note that this example is identical to the two vertex graph example in
Figure~\ref{fig1} with $p=4$, but shifted by $(-1,-1)$.} \label{fig3}
\end{figure}
We include in Figure~\ref{fig3} a representation of the divisors $x \in \R^{2}$ with $\ell(x)=0$ 
for the example used in the proof of Theorem~\ref{NonGraph}.  The plot is identical to that of a two
vertex graph with $p=4$ but is shifted by $-1$ in each direction.   

It is also possible to produce non-graph examples by using more generators for $\mathcal{N}$.  
In Figure~\ref{fig4}, the divisors $x \in \R^{2}$ with $\ell(x)=0$ are shown where 
$\mathcal{N}$ is generated by two points $\nu_{1}=(0,4)$ and $\nu_{2}=(1,3)$, 
using $H=<(-3,3)>$ and $\kappa=(0,0)$.  
\begin{figure}
\includegraphics[width=3in]{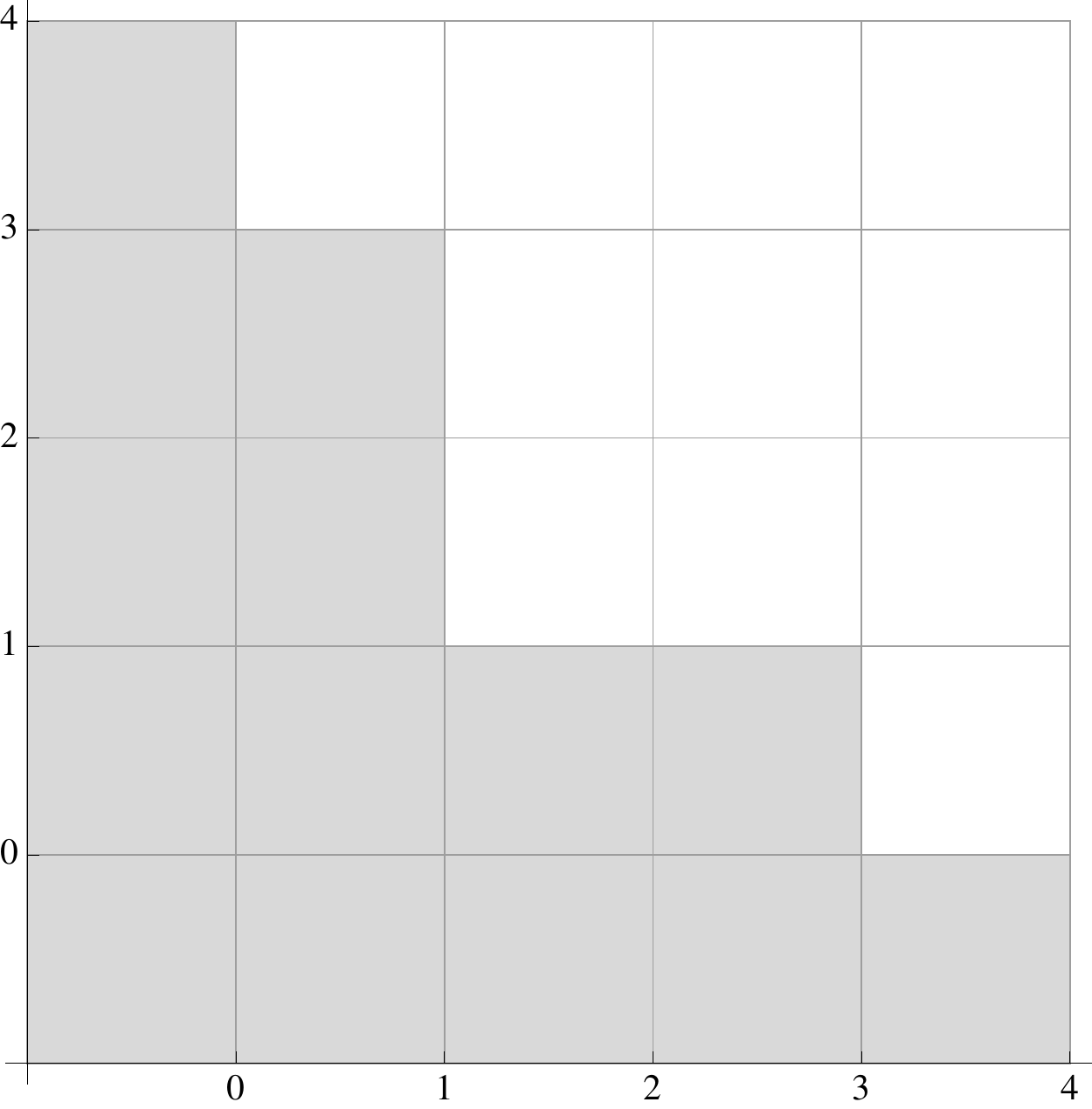}
\caption{Divisors $x \in \R^{2}$ with $\ell(x)=0$ for a non-graph example with $\mathcal{N}$ 
generated by $(0,4)$ and $(1,3)$, $H=<(-3,3)>$ and $\kappa=(0,0)$.} \label{fig4}
\end{figure}



\begin{thebibliography}{99}

\bibitem{BN}
Matthew Baker and Serguei Norine,
Riemann-Roch and Abel-Jacobi Theory on a Finite Graph,
\emph{Advances in Mathematics} 215 (2007) 766-788.

\bibitem{JM09}
Rodney James and Rick Miranda,
A Riemann-Roch theorem for edge-weighted graphs,
arXiv:0908.1197v2 [math.AG], 2009.

\bibitem{JM11}
Rodney James and Rick Miranda,
Linear Systems on Edge-Weighted Graphs,
arXiv:1105.0227v1 [math.AG], 2011.

\end{thebibliography}
\end{document}